\documentclass[11pt]{amsart}

\usepackage{lmodern}
\usepackage[T1]{fontenc}
\usepackage{amsmath,amssymb,amsthm}
\usepackage{booktabs}
\usepackage[margin=1.15in]{geometry}
\usepackage{microtype}
\usepackage{tikz}
\usetikzlibrary{calc}
\usepackage[hidelinks]{hyperref}

\tikzset{
  cell/.style  = {draw, minimum size=11mm, inner sep=0pt, font=\footnotesize},
  dom/.style   = {draw, very thick},
  conn/.style  = {draw, dashed},
  pt/.style    = {circle, fill, inner sep=1.15pt},
  leaf/.style  = {circle, draw, fill=white, inner sep=1.15pt},
  strip/.style = {gray!55, dotted, thick},
}

\theoremstyle{plain}
\newtheorem{theorem}{Theorem}[section]
\newtheorem{proposition}[theorem]{Proposition}
\newtheorem{lemma}[theorem]{Lemma}
\newtheorem{corollary}[theorem]{Corollary}
\theoremstyle{definition}
\newtheorem{definition}[theorem]{Definition}
\newtheorem{remark}[theorem]{Remark}

\newcommand{\code}[1]{\texttt{#1}}
\newcommand{\Av}{\operatorname{Av}}
\newcommand{\gr}{\operatorname{gr}}
\newcommand{\Lam}{\Lambda}
\newcommand{\Qc}{Q}

\begin{document}

\title[A new lower bound for the growth rate of Av(1324)]
      {A new lower bound for the growth rate of
       \texorpdfstring{$\Av(1324)$}{Av(1324)}}
\author{Charles C. Norton}
\email{cnorton2@binghamton.edu}
\date{September 2026}

\begin{abstract}
The growth rate of $\Av(1324)$ is the last unknown Stanley--Wilf limit of a
length-four pattern. Since 2020 the best rigorous lower bound has been the
$10.271012$ of Bevan, Brignall, Elvey Price and Pantone; we raise it to
$10.412263$. Their construction lays $1324$-avoiders along an infinite
staircase of two-cell blocks alternating with single connecting cells, and
forbids a $1324$ by a local condition on how the points of a block interleave
with the components of the cell beside it. Certain points of a block, its
leaves, may be exempted from that condition. They exempt them on one axis only,
and report that they could not count the possibilities when both are exempted.
This paper exempts them on both axes and counts the result. The second
relaxation is their own argument read in the other axis, and the Harris
correlation inequality bounds the joint count below by the product of the two
counts taken separately, at the cost of one factor of the component generating
function.
\end{abstract}

\maketitle

\section{Introduction}

Write $\Av(1324)$ for the class of permutations containing no subsequence order
isomorphic to $1324$, and $\Av_n(1324)$ for its members of length $n$. By
Marcus and Tardos \cite{marcustardos} together with Arratia \cite{arratia} the
limit
\[
  \gr(\Av(1324)) \;=\; \lim_{n \to \infty} |\Av_n(1324)|^{1/n}
\]
exists. It is the one Stanley--Wilf limit of a length-four pattern that remains
unknown; the other two Wilf classes were settled in the 1990s by Gessel
\cite{gessel} and B\'ona \cite{bona1342}. The counting sequence is
\code{A061552} \cite{oeis}, known exactly for $n \le 50$ \cite{cgzj}. Its
analysis suggests $|\Av_n(1324)| \sim B \mu^n \mu_1^{\sqrt n} n^{g}$ with
$\mu = 11.600 \pm 0.003$. Proved, that form would put the sequence outside the
$P$-recursive ones: by Garrabrant and Pak \cite[Thm.~9]{garrabrantpak} an
integer $P$-recursive sequence of at most exponential growth carries no
stretched-exponential factor.

The upper bound has been $13.5$ and the lower
$10.271012$ since the work of Bevan, Brignall, Elvey Price and Pantone
\cite{bbep}, whose structural characterisation of $\Av(1324)$ as a subclass of
an infinite staircase grid class underlies everything below. Before them the
upper bound came down from B\'ona's $288$ \cite{bona288} through
\cite{cjs, bona1393} to $13.74$ \cite{bona1374}, and the lower rose from $9$
\cite{bonalayered} through \cite{albert} to $9.81$ \cite{bevan}. We refer to that
paper throughout as BBEP and adopt its notation, citing its results by the
sequential numbering of the journal version. The arXiv version numbers by
section instead: its Theorem 3.1, Theorem 5.1 and Lemma 7.4 are the Theorem 2,
Theorem 9 and Lemma 19 cited here.

Sharper values are available under hypotheses of their own. Frankl\'in
\cite{franklin} encodes $1324$-avoiders as walks in a directed graph, making
the growth rate a spectral radius, and a weighted quotient of it gives
$10.418$, conditional on the conjecture that the weighted walks never outnumber
the unweighted ones. Bostan, Elvey
Price, Guttmann and Maillard \cite{bepgm} give numerical evidence that the
counting sequence is a Stieltjes moment sequence; on that hypothesis the
coefficients of the associated continued fraction give $10.302$, and $10.607$
if two of their subsequences are assumed increasing. On the upper side the
conjecture of Claesson, Jel\'inek and Steingr\'imsson \cite{cjs} on the
$1324$-avoiders with a fixed number of inversions would give
$e^{\pi\sqrt{2/3}} \approx 13.002$; Linusson and Verkama \cite{lv} prove half
of it, and Claesson, Linusson, Ulfarsson and Verkama \cite{cluv} exhibit the
first nontrivial inversion-monotone sets containing $1324$. The bound proved
here is unconditional.

BBEP close their paper by naming three ways their bound might be improved. The
second is the relaxation of their interleaving rule in the vertical direction,
of which they write that they ``have not been able to determine a lower bound
for the number of possibilities'' and that ``it seems likely that the
one-dimensional solution in which leaves are distributed equitably between the
strips does not carry over to interleaving in two directions.'' This paper
answers that question. The relaxation itself, Lemma~\ref{lem:valid} below, is
their own one-directional argument read in the other axis. What was open is the
count, which Theorem~\ref{thm:harris} supplies. At a fixed value of the size
variable a connecting cell is a sequence of independent components, and both
interleaving counts are increasing functions of the component sizes; Harris'
inequality therefore bounds the joint count below by the product of its two
marginals. Each neighbour then sits in its own copy of the one-directional
machinery and is minimised there, so no two-directional analogue of the
equitable solution is needed. The separation costs one factor of the component
generating function per component. It is the only step at which the two counts
are taken apart. Their first problem, the distribution of $k$-leaf strips, and
their third, the enumeration of three-celled trominoes, stand as they left
them.

\begin{theorem}\label{thm:main}
$\gr(\Av(1324)) \ge 10.412263$.
\end{theorem}

The proof occupies Sections~\ref{sec:relax} to~\ref{sec:bound}.
Section~\ref{sec:background} fixes the staircase, the interleaving operators
and the equitable profile, all of them BBEP's. Section~\ref{sec:relax} is the
new material: it proves the second relaxation valid, closes the whole
decomposition under it, and supplies the count. Section~\ref{sec:exponent}
assembles the exponent and states the proposition that carries a positive value
of it to a bound on the growth rate. Section~\ref{sec:bound} carries out the
evaluation and checks it against two values known independently.

Two earlier versions, \code{arXiv:2608.20292v1} and~\code{v2}, reported
$10.617$ and $10.466425$, from a joint transfer operator for the two neighbours
of a connecting cell and from tilting the domino ensemble. Both are withdrawn
as unproved and neither should be cited; neither construction is used here, the
strip profile entering only through BBEP's equitable minimiser. Everything the
theorem depends on is the single evaluation of Table~\ref{tab:cert}: three
functions displayed in closed form in Section~\ref{sec:bound}, read with their
$q$-derivatives at three points, and a tangent-line bound on a convex function
of one variable.

\section{Background}\label{sec:background}

\subsection{The staircase}

The descending $(\Av(213), \Av(132))$ staircase is the infinite grid class whose
diagonal cells avoid $213$, whose subdiagonal cells avoid $132$, and whose other
cells are empty. BBEP show by an explicit greedy gridding that $\Av(1324)$ is
contained in it \cite[Prop.~1]{bbep}. Index the cells $k = 1, 2, \ldots$
descending from the top left. Consecutive cells meet in one of two ways. An odd
cell and its successor share a column, interleave freely in position, and the
odd cell holds the larger values; we call that a $(V)$ pair. An even cell and
its successor share a row, interleave freely in value, and the even cell comes
first in position; we call that an $(H)$ pair. Cells at distance at least two
are separated in both coordinates. Figure~\ref{fig:staircase} shows the first
nine cells together with the decomposition of Section~\ref{sec:refined}, which
repeats with period six.

\begin{figure}[t]
\centering
\begin{tikzpicture}[x=13mm, y=-13mm]
  \foreach \c/\r/\k/\cls in {%
    1/1/1/{\Av(213)}, 1/2/2/{\Av(132)}, 2/2/3/{\Av(213)},
    2/3/4/{\Av(132)}, 3/3/5/{\Av(213)}, 3/4/6/{\Av(132)},
    4/4/7/{\Av(213)}, 4/5/8/{\Av(132)}, 5/5/9/{\Av(213)}}
    {\node[cell] (n\k) at (\c,\r) {};
     \node[font=\scriptsize] at ($(\c,\r)+(0,0.10)$) {$\cls$};
     \node[font=\scriptsize\bfseries, anchor=north west]
       at ($(\c,\r)+(-0.40,-0.40)$) {\k};}
  \draw[dom] ($(1,1)+(-0.48,-0.48)$) rectangle ($(1,2)+(0.48,0.48)$);
  \draw[dom] ($(2,3)+(-0.48,-0.48)$) rectangle ($(3,3)+(0.48,0.48)$);
  \draw[dom] ($(4,4)+(-0.48,-0.48)$) rectangle ($(4,5)+(0.48,0.48)$);
  \draw[conn] ($(2,2)+(-0.48,-0.48)$) rectangle ($(2,2)+(0.48,0.48)$);
  \draw[conn] ($(3,4)+(-0.48,-0.48)$) rectangle ($(3,4)+(0.48,0.48)$);
  \draw[conn] ($(5,5)+(-0.48,-0.48)$) rectangle ($(5,5)+(0.48,0.48)$);
  \draw[dom]  (5.19,1.43) rectangle (5.35,1.77);
  \draw[dom]  (5.10,2.06) rectangle (5.44,2.22);
  \draw[conn] (5.17,2.60) rectangle (5.37,2.80);
  \node[anchor=west, font=\footnotesize] at (5.55,1.6)  {vertical domino};
  \node[anchor=west, font=\footnotesize] at (5.55,2.14) {horizontal domino};
  \node[anchor=west, font=\footnotesize] at (5.55,2.7)  {connecting cell};
\end{tikzpicture}
\caption{The first nine cells of the descending
$(\Av(213), \Av(132))$ staircase, with the decomposition of
Section~\ref{sec:refined}. Cells $6j+1, 6j+2$ form a vertical domino and
$6j+4, 6j+5$ a horizontal one; cells $6j+3$ and $6j+6$ connect them. Cells
$1$ and $2$ share a column, a $(V)$ pair; cells $2$ and $3$ share a row, an
$(H)$ pair.}
\label{fig:staircase}
\end{figure}
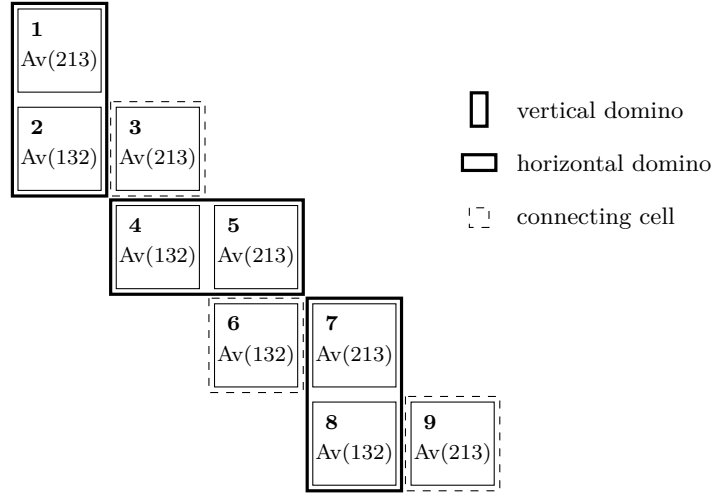

A \emph{domino} is a two-cell gridded permutation in
$\operatorname{Grid}^{\#}\bigl(\begin{smallmatrix}\Av(213)\\ \Av(132)\end{smallmatrix}\bigr)$
avoiding $1324$. BBEP enumerate them: the number of $n$-point dominoes is
$2(3n+3)!/((n+2)!(2n+3)!)$, sequence \code{A000139} \cite{oeis}, of growth rate
$27/4$ \cite[Thm.~2]{bbep}, and balanced dominoes, those with the same number of
points in each cell, have the same growth rate \cite[Prop.~7]{bbep}.

Two facts about occurrences are used throughout. The first is a locality
statement which BBEP give as an observation.

\begin{lemma}[locality]\label{lem:L}
In any staircase-gridded permutation the four points of an occurrence of $1324$
lie in two adjacent cells, two points in each.
\end{lemma}

\begin{proof}
Write the occurrence $P_1 P_2 P_3 P_4$ in position order with values
$v_1 < v_3 < v_2 < v_4$ and let $c_t$ be the cell of $P_t$. Of the six pairs
exactly one, $(P_2, P_3)$, is an inversion. Two points in cells at distance at
least two form an inversion with the point in the lower-indexed cell earlier and
larger, so every one of the five ascent pairs has its cells at distance at most
one; hence $|c_1 - c_4| \le 1$ and each of $c_2, c_3$ lies within one of both.

Suppose $\max_t c_t - \min_t c_t \ge 2$. The only pair permitted at distance two
is $(c_2, c_3)$, so the extremes are attained there, $c_1 = c_4$ is the unique
index within one of both, and $|c_2 - c_3| = 2$; the separation condition applied
to the inversion puts $P_2$ in the lower-indexed cell, so $c_2 = t$,
$c_1 = c_4 = t+1$ and $c_3 = t+2$. For $t$ odd, cells $t$ and $t+1$ form a
$(V)$ pair, so $v_2 > v_4$, against $v_2 < v_4$; for $t$ even they form an
$(H)$ pair, so $P_2$ precedes $P_1$, against the position order. Neither
parity is possible.

All four cells therefore lie in a window of width one. A single cell is
impossible, since $1324$ contains $132$ on its first three entries and $213$ on
its last three while every cell avoids one of the two. In a $(V)$ pair a split
with one point below forces that point to be $P_1$, leaving $P_2P_3P_4$ above in
the pattern $213$; a split with one point above forces $P_4$, leaving
$P_1P_2P_3$ below in the pattern $132$. Case $(H)$ follows by reflection in the
anti-diagonal, which maps the staircase to itself, exchanges $\Av(213)$ and
$\Av(132)$, and fixes $1324$.
\end{proof}

The second is the criterion for a domino, immediate from the case analysis above.

\begin{lemma}[domino criterion]\label{lem:crit}
Split a word by value at a threshold. If the low cell avoids $132$ and the high
cell avoids $213$, then an occurrence of $1324$ is exactly an ascent of the low
cell interleaved in position with an ascent of the high cell.
\end{lemma}

Finally we isolate the property of skew decompositions that the interleaving
rules use. For any permutation $w = c_1 \ominus \cdots \ominus c_r$ the
components occupy consecutive position and value intervals with earlier
components holding larger values, so a point of $c_i$ and a later point of $c_j$
with $i < j$ form a descent, and
\begin{equation}\label{eq:A}
  \text{every ascent of } w \text{ lies inside a single skew component.}
\end{equation}
This holds for every permutation; we use it for the connecting cells, which lie
in $\Av(132)$ or $\Av(213)$.

\subsection{Leaves, strips, and the interleaving count}\label{sec:strips}

Leaves and strips are defined for all four kinds of domino cell at once. A
\emph{leaf} of a cell is a right-to-left maximum if the cell avoids $213$ and a
left-to-right minimum if it avoids $132$, which covers the four: a vertical
domino occupies a column, its upper cell avoiding $213$ and its lower avoiding
$132$, and a horizontal domino occupies a row, its left cell avoiding $132$ and
its right cell avoiding $213$. The $r$ non-leaves of a cell cut it into $r+1$
\emph{strips}, along the axis on which the cell meets the connecting cell beside
it: a cell of a vertical domino shares a row with that cell and interleaves with
it in value, so its strips are horizontal and are cut by value; a cell of a
horizontal domino shares a column and interleaves in position, so its strips are
vertical and are cut by position. Lemma~\ref{lem:reflect} identifies the two
readings, so everything below may be and is written in the vertical case. A
strip with $j$ leaves is a \emph{$j$-leaf strip}, and one with none is
\emph{empty}; the \emph{strip profile} of a cell records how many strips it has
of each kind. Figure~\ref{fig:strips} is an example.

\begin{figure}[t]
\centering
\begin{tikzpicture}[x=7.5mm, y=7.5mm]
  \foreach \y in {1,3,4} \draw[strip] (0.3,\y) -- (6.5,\y);
  \draw[->] (0.3,0.3) -- (6.0,0.3) node[below, font=\scriptsize] {position};
  \draw[->] (0.3,0.3) -- (0.3,5.8) node[above, font=\scriptsize] {value};
  \node[pt]   at (1,3) {}; \node[pt]   at (2,4) {};
  \node[leaf] at (3,5) {}; \node[pt]   at (4,1) {};
  \node[leaf] at (5,2) {};
  % one rule per strip, spanning it, so each label names a band and not a line
  \draw[gray!55] (6.65,0.35) -- (6.65,0.95);
  \draw[gray!55] (6.65,1.05) -- (6.65,2.95);
  \draw[gray!55] (6.65,3.05) -- (6.65,3.95);
  \draw[gray!55] (6.65,4.05) -- (6.65,5.70);
  \node[font=\scriptsize, anchor=west] at (6.8,0.65) {$0$-leaf};
  \node[font=\scriptsize, anchor=west] at (6.8,2.0)  {$1$-leaf};
  \node[font=\scriptsize, anchor=west] at (6.8,3.5)  {$0$-leaf};
  \node[font=\scriptsize, anchor=west] at (6.8,4.8)  {$1$-leaf};
\end{tikzpicture}
\caption{The $\Av(213)$ cell $34512$. Its right-to-left maxima, the leaves, are
drawn open; the three non-leaves are filled and cut the cell into four
horizontal strips, two empty and two carrying one leaf each, the rules on the
right marking the four. Each dotted line lies at the value of the non-leaf that
draws it. A cell of the lower kind is the same picture reflected in the
anti-diagonal, its leaves the left-to-right minima.}
\label{fig:strips}
\end{figure}
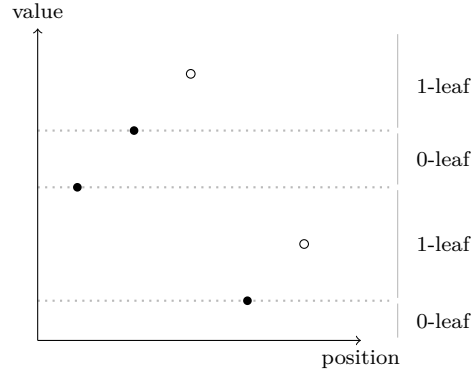

The generating function for a connecting cell counted by points and components is
\begin{equation}\label{eq:H}
  H(z,q) \;=\; \frac{1}{1 - q\Qc(z)} \;=\;
  \frac{2}{2 - q\bigl(1 - \sqrt{1-4z}\bigr)},
  \qquad \Qc(z) \;=\; \tfrac12\bigl(1 - \sqrt{1-4z}\bigr),
\end{equation}
with $\Qc$ the generating function of a single skew indecomposable component.
The generating function for the possibilities in a strip with $j$ leaves is
$H_j(z,q) = \Omega_j[H(z,q)]$, where $\Omega_j$ is the linear operator
\begin{equation}\label{eq:Omega}
  \Omega_j[z^n] \;=\; \binom{n+j}{j} z^n,
  \qquad\text{equivalently}\qquad
  \Omega_j[F(z)] \;=\; \frac{1}{j!}
  \frac{\partial^j}{\partial z^j}\bigl(z^j F(z)\bigr).
\end{equation}
The count factors over the strips. A non-leaf of the domino cell must sit
between two consecutive
components of the connecting cell, so the non-leaves cut the component sequence
into as many consecutive blocks as there are strips, one block per strip and in
the same order; within a strip the leaves are placed freely among the components
of its own block, and nothing links one strip to another. The count is therefore
$\prod_i H_{a_i}(z,q)$ over the strip sizes $a_i$, a product depending on the
multiset $\{a_i\}$ and not on the order in which the strips appear, so the
exponent may be written in terms of the profile alone.

BBEP's Proposition 18 states that the sequence $H_0, H_1, H_2, \ldots$ is
log-convex in the coefficientwise order. Among all profiles with a given number
of strips and a given number of leaves, an equitable distribution of the leaves
among the strips therefore minimises that product. That is how their Theorem 16
bounds the product without determining the profile, and how we bound it here.
The construction asks of a cell only that its leaves and its empty strips be at
least as many as two thresholds. Lemma~\ref{lem:equitbound} carries the minimum
from a cell's own counts down to those thresholds.

Two densities of BBEP's fix the equitable profile. Their Proposition 10 gives
the leaf density of a domino cell as $5/9$ per point and their Proposition 12
the density of medial empty strips as $5/27$. The following reconciles the
second with the empty-strip count itself.

\begin{remark}[medial and total empty strips]\label{rem:medial}
Call an empty strip \emph{medial} if it is neither the lowest nor the highest of
its cell's strips. The highest strip of a top cell holds the cell's largest
value, which is a right-to-left maximum and so a leaf, so that strip is never
empty; every empty strip but possibly the lowest is therefore medial. The bottom cell is
the same under the $180^\circ$ rotation, with the two ends exchanged, and the
cell $12$ of $\Av(132)$ shows its highest strip really can be empty. Either way
the two counts differ by $0$ or $1$ per cell, which is $O(1)$ against
$\Theta(n)$ points and vanishes from the density. What is reconciled here is the
constant and not a threshold: BBEP's family, Lemma~\ref{lem:equitbound} and
Proposition~\ref{prop:transfer} are all stated in empty strips outright, and no
count is converted from one reading to the other.
\end{remark}

\noindent So the $5/27$ of BBEP's Proposition 12 may be read as a density of
empty strips outright, which is the form the threshold $\beta < 5/27$ takes
below.

Fix densities $\alpha$ and $\beta$ for the leaves and the empty strips. Since
$r$ non-leaves give $r+1$ strips, a cell on $m$ points with $\alpha m$ leaves
carries $(1-\alpha)m$ strips, of which $\beta m$ are empty. BBEP show
\cite[\S7.2]{bbep} that for $\alpha \in [11/20, 5/9)$ and
$\beta \in [7/40, 5/27)$ and $m \ge 32$ the equitable distribution of the leaves
among the non-empty strips uses only strips with two and with three leaves, in
numbers
\begin{equation}\label{eq:equit}
  e_0(m) = \lceil\beta m\rceil + 1, \quad
  e_2(m) = 3m - 4\lceil\alpha m\rceil - 3\lceil\beta m\rceil, \quad
  e_3(m) = 3\lceil\alpha m\rceil + 2\lceil\beta m\rceil - 2m ,
\end{equation}
the solutions of $2e_2 + 3e_3 = \lceil\alpha m\rceil$ and
$e_0 + e_2 + e_3 = m - \lceil\alpha m\rceil + 1$.

All four endpoints are BBEP's, and the two lower ones are where the equitable
distribution changes form. The average number of leaves per non-empty strip is
$\alpha/(1-\alpha-\beta)$, at least $2$ exactly when $3\alpha + 2\beta \ge 2$
and at most $3$ exactly when $4\alpha + 3\beta \le 3$; on the rectangle the
first holds with equality only at the corner $(11/20,\, 7/40)$ and the second
throughout, since $4\alpha + 3\beta < 25/9$ there. So the equitable
distribution uses parts of size two and three and no others. Per point of the cell the
three densities are
\begin{equation}\label{eq:fprofile}
  f_0 = \beta, \qquad f_2 = 3 - 4\alpha - 3\beta, \qquad
  f_3 = 3\alpha + 2\beta - 2 ,
\end{equation}
which at the limiting values $(\alpha,\beta) = (5/9, 5/27)$ read
$(5/27,\, 2/9,\, 1/27)$. We write $f$ for the profile \eqref{eq:fprofile}
throughout.

\subsection{The refined construction}\label{sec:refined}

BBEP decompose the staircase into an alternating sequence of dominoes and single
\emph{connecting} cells, so that each period contributes exactly two
block-cell-to-connecting-cell adjacencies. Avoiding $1324$ is then guaranteed by
a local interleaving rule. In their Theorem 9 the rule is that every point of a
domino cell lies between two consecutive skew components of the adjacent
connecting cell, which gives $81/8 = 10.125$. In their Theorem 16 the rule is
relaxed for the cells \emph{horizontally} adjacent to a connecting cell, where
only the non-leaves need lie between components. A leaf of such a cell has
nothing to its upper right, so it cannot act as the $2$ of an occurrence with a
$4$ beyond it; in the opposite orientation the mirrored argument bars it from
acting as the $3$. That relaxation gives $10.271012$.

Their exponent is
\begin{equation}\label{eq:phi}
  \Phi(z,\gamma,\kappa) \;=\; (1+\gamma)\log\frac{27z}{4}
  \;-\; \kappa \log q_0 \;+\; \sum_j f_j \log H_j(z,q_0)
  \;+\; \mathcal{E}(\gamma,\kappa),
\end{equation}
in which $\mathcal{E}$ is the exponential growth rate of the binomial
coefficient that counts the unrelaxed interleavings,
\begin{equation}\label{eq:ent}
  \mathcal{E}(\gamma,\kappa) \;=\;
  (\gamma+\kappa)\log(\gamma+\kappa) - \gamma\log\gamma - \kappa\log\kappa
  \;=\; \lim_{m\to\infty}\frac1m\log\binom{\gamma m+\kappa m}{\kappa m},
\end{equation}
$\gamma$ is the ratio of the two cell sizes, $\kappa$ the ratio of a connecting
cell's component count to a cell size, and $q_0$ the saddle solving
$\sum_j f_j\,q\,(\partial_q H_j)/H_j = \kappa$; the bound is $\gr \ge 1/z^*$ at
the root of $\Phi = 0$.

Their $\gamma$ and $\kappa$ can be eliminated in closed form: stationarity in
\eqref{eq:phi} reads $\gamma/(\gamma+\kappa) = 27z/4$ and
$\kappa/(\gamma+\kappa) = 1/q_0$, so that $q_0 = 4/(4-27z)$. That collapse uses
the binomial term of \eqref{eq:phi} and holds for that exponent only;
Section~\ref{sec:exponent} replaces the term, and
Section~\ref{sec:controls} checks the replacement against their published
numbers.

\section{The second relaxation and its count}\label{sec:relax}

Section~\ref{sec:valid} proves the second relaxation valid and closes the whole
decomposition under both. Section~\ref{sec:harris} is the count.

\subsection{Validity}\label{sec:valid}

\begin{lemma}\label{lem:valid}
Let a $(V)$ pair consist of a connecting cell and a domino cell sharing a column.
If every non-leaf of the domino cell lies between two consecutive skew components
of the connecting cell, then the pair contains no occurrence of $1324$. Leaves
may be placed arbitrarily.
\end{lemma}

\begin{proof}
By Lemma~\ref{lem:L} an occurrence in the pair splits two and two, and by
Lemma~\ref{lem:crit} it is an ascent of the lower cell interleaved in position
with an ascent of the upper cell:
\[
  P_1 < P_2 < P_3 < P_4 \text{ in position}, \qquad v_1 < v_3 < v_2 < v_4,
\]
with $P_1, P_3$ in the lower cell and $P_2, P_4$ in the upper.

Suppose the connecting cell is the upper one. Then $(P_2, P_4)$ is one of its
ascents, so by \eqref{eq:A} both points lie in a single skew component, and $P_3$
sits positionally strictly between them, hence inside that component. Now $P_1$
precedes $P_3$ and is smaller, so $P_3$ is not a left-to-right minimum of the
lower cell: it is a non-leaf.

Suppose instead the connecting cell is the lower one. Then $(P_1, P_3)$ is its
ascent, both points lie in one component by \eqref{eq:A}, and $P_2$ sits
positionally inside that component. Now $P_4$ follows $P_2$ and is larger, so
$P_2$ is not a right-to-left maximum of the upper cell: again a non-leaf.

In either case the point driven inside a component is a non-leaf of the domino
cell, so confining the non-leaves to the gaps forbids the occurrence.
\end{proof}

\begin{remark}\label{rem:sharp}
The condition ``no non-leaf is straddled by an ascent of the connecting cell''
looks weaker than confining the non-leaves to the gaps between components, but
the two are equivalent: skew indecomposability says exactly that for every
interior slot $t$ there are $i < t \le j$ with $\pi_i < \pi_j$, an ascent
straddling $t$. The rule is nonetheless weaker than the exact condition, which
asks of a straddled non-leaf that it \emph{also} have an earlier smaller point
of its own cell. The gap rule is that first half alone, and what the second
half would buy is open.
\end{remark}

Lemma~\ref{lem:valid} concerns one adjacent pair. In the decomposition the two
relaxations meet on a single connecting cell, which sits in an $(H)$ adjacency
on one axis and a $(V)$ adjacency on the other. They do not interact.

\begin{corollary}\label{cor:closure}
Take the decomposition of Section~\ref{sec:refined} and require of every cell of
a domino that its non-leaves lie between consecutive skew components of the
adjacent connecting cell, leaving its leaves free. Then the gridded permutation
avoids $1324$.
\end{corollary}

\begin{proof}
By Lemma~\ref{lem:L} the four points of an occurrence lie in two adjacent cells,
two in each, so an occurrence is confined to some pair $\{k, k+1\}$ and none
draws points from both neighbours of a connecting cell. In the decomposition
every such pair is of one of three kinds. The two cells of a common domino form
a domino, which avoids $1324$ by definition. A cell of a vertical domino
together with the connecting cell beside it forms an $(H)$ pair, closed by
BBEP's rule \cite[\S7]{bbep}, whose leaves are free there. A cell of a
horizontal domino together with the connecting cell above or below it forms a
$(V)$ pair, closed by Lemma~\ref{lem:valid}, whose leaves are free by that
lemma. No two cells of distinct dominoes are adjacent, since the dominoes occupy
the cells $6j+1, 6j+2$ and $6j+4, 6j+5$ and the connecting cells the remaining
ones, so the list is complete.
\end{proof}

\subsection{The Harris separation}\label{sec:harris}

Relaxing both directions destroys the factorisation that makes BBEP's count
tractable. For a fixed connecting cell $C$ the two interleavings are
independent, since the horizontal neighbour meets $C$ in value while the
vertical one meets it in position, so the count is
$\sum_C N_{\mathrm{h}}(C; D_{\mathrm{h}})\,N_{\mathrm{v}}(C; D_{\mathrm{v}})$
for the two domino cells $D_{\mathrm{h}}$ and $D_{\mathrm{v}}$ that meet $C$.
BBEP can factor this because their unrelaxed
$N_{\mathrm{v}}(C; D_{\mathrm{v}}) = \binom{m+c}{c}$ depends on $C$ only through
its number of components $c$; once relaxed it does not, so the sum no longer
splits.

The obstruction is removed by a correlation inequality: if $X_1, \ldots, X_c$
are independent and $F$ and $G$ are both increasing functions of
$(X_1, \ldots, X_c)$, then $\mathbb{E}[FG] \ge \mathbb{E}[F]\,\mathbb{E}[G]$.
Harris proves this for Bernoulli coordinates \cite{harris}, and Grimmett
\cite[Thm.~2.4]{grimmett} gives it in the form used here, for a product measure
on a product of totally ordered spaces; it is the product-measure case of
Fortuin, Kasteleyn and Ginibre. In one variable it is Chebyshev's association
inequality, and it follows for $c$ of them by conditioning and inducting.

\begin{definition}\label{def:placement}
Let $C$ be a connecting cell with skew components of sizes
$\mathbf{w} = (w_1, \ldots, w_c)$, each $w_i \ge 1$, and $|C| = \sum_i w_i$
points, and let $D$ be a domino cell with a distinguished set $S$ of its points.
A \emph{placement} of $D$ against $C$ is a weakly increasing map from the points
of $D$ to $\{0, 1, \ldots, |C|\}$, counting for each point how many points of
$C$ precede it. The points of $D$ are taken in the order along which the two
cells interleave: by position for the vertical neighbour of
Lemma~\ref{lem:valid}, which shares a column with $C$, and by value for the
horizontal one, which shares a row. Either is a total order on the points of
$D$, along which \emph{weakly increasing} is meant. A placement is
\emph{admissible} when every point of $S$ is sent into
\[
  G(\mathbf{w}) \;=\; \{0,\; \sigma_1,\; \sigma_2,\; \ldots,\; \sigma_c\},
  \qquad \sigma_i = w_1 + \cdots + w_i ,
\]
the positions lying between consecutive components together with the two
exterior positions, and $N(C;D)$ is the number of admissible placements, the
set $S$ being understood. Taking $S$ to be the non-leaves of $D$ is
the rule of Lemma~\ref{lem:valid} and of BBEP's Section 7.1 written out. The
rule for the horizontal neighbour differs only in the axis along which the
components are read, and everything below applies to it verbatim.

The $|S|$ points of $S$ cut $D$ into $|S|+1$ \emph{strips} along that same
axis, exactly as its non-leaves do in Section~\ref{sec:strips}, which is the
case $S = \{\text{non-leaves of } D\}$; and the factorisation
\[
  \sum_{C} z^{|C|} q^{c(C)} N(C;D) \;=\; \prod_i H_{a_i}(z,q),
\]
over the strip sizes $a_i$, the sum running over all connecting cells $C$ with
$c(C)$ components, holds for every $S$. Its proof in Section~\ref{sec:strips}
uses only that the points of $S$ are confined to $G(\mathbf{w})$ and the rest
are free. Only $S = \{\text{non-leaves}\}$ occurs outside
Lemma~\ref{lem:equitbound}.
\end{definition}

\begin{lemma}\label{lem:mono}
$N(C;D)$ depends on $C$ only through $\mathbf{w}$, so we may write
$N(\mathbf{w};D)$ for it; and $N(\mathbf{w};D) \le N(\mathbf{w}';D)$ whenever
$\mathbf{w} \le \mathbf{w}'$ coordinatewise.
\end{lemma}

\begin{proof}
Only $G(\mathbf{w})$ enters Definition~\ref{def:placement}, and $G(\mathbf{w})$
is determined by $\mathbf{w}$, so the internal arrangement of each component is
free and the count is a function of $\mathbf{w}$ alone.

For the inequality it is enough to raise one coordinate by one and iterate. Let
$\mathbf{w}'$ agree with $\mathbf{w}$ except that $w'_j = w_j + 1$, and let
$\psi$ on $\{0, \ldots, |C|\}$ fix $x \le \sigma_{j-1}$ and send $x$ to $x+1$
beyond. Then $\psi$ is strictly increasing and
$\psi(G(\mathbf{w})) = G(\mathbf{w}')$ exactly: it fixes $0$ and each
$\sigma_i$ with $i \le j-1$, and raises by one each $\sigma_i$ with $i \ge j$,
those exceeding $\sigma_{j-1}$ because $w_j \ge 1$, which is what passing to
$\mathbf{w}'$ does to them. Composing with $\psi$ therefore carries an
admissible placement against $\mathbf{w}$ to one against $\mathbf{w}'$, weak
increase and the non-leaf condition both being preserved, and injectively,
since $\psi$ is injective.
\end{proof}

The two neighbours read the components of $C$ in opposite orders, the horizontal
one in value and the vertical one in position, so the two counts are functions
of the same vector read forwards and backwards. Reversing the arguments of a
function non-decreasing in each leaves it so, and the law below is a product
measure, invariant under the reversal; no rearrangement is needed.

\begin{theorem}\label{thm:harris}
Fix $z \in (0,\tfrac14)$ and two domino cells with their non-leaves
distinguished, $D_{\mathrm{h}}$ meeting the connecting cell in value and
$D_{\mathrm{v}}$ meeting it in position. Then, for every $c \ge 1$, the sums
running over the connecting cells with $c$ components,
\[
  \sum_C z^{|C|}N_{\mathrm{h}}(C; D_{\mathrm{h}})\,
  N_{\mathrm{v}}(C; D_{\mathrm{v}}) \;\ge\;
  \frac{\bigl(\sum_C z^{|C|}N_{\mathrm{h}}(C; D_{\mathrm{h}})\bigr)
        \bigl(\sum_C z^{|C|}N_{\mathrm{v}}(C; D_{\mathrm{v}})\bigr)}
  {\Qc(z)^{\,c}} .
\]
\end{theorem}

\begin{proof}
By Lemma~\ref{lem:mono} the count $N(C;D)$ is a function of
$\mathbf{w} = (|\gamma_1|, \ldots, |\gamma_c|)$ alone, where
$(\gamma_1, \ldots, \gamma_c)$ is the sequence of skew indecomposables making up
$C$. Let $a_k$ be the number of skew indecomposable components on $k$ points, so
that $\sum_{k \ge 1} a_k z^k = \Qc(z)$, put
$\mu_z(k) = a_k z^k/\Qc(z)$ for $k \ge 1$, a probability distribution, and let
$X_1, \ldots, X_c$ be independent with law $\mu_z$. A connecting cell of $c$
components and $N$ points has weight $z^{N}$, and the point count is
unconstrained, so for any $\Psi$ of the component sizes
\[
  \sum_C \Psi(w_1, \ldots, w_c)\,z^{|C|}
  \;=\; \Qc(z)^{c}\; \mathbb{E}\bigl[\Psi(X_1, \ldots, X_c)\bigr],
\]
the sum running over connecting cells with $c$ components; taking $\Psi = 1$
gives $\sum_C z^{|C|} = \Qc(z)^c$. By Lemma~\ref{lem:mono} and the paragraph
before the statement, both $N_{\mathrm{h}}$ and $N_{\mathrm{v}}$ are
non-decreasing functions of $(X_1, \ldots, X_c)$ in the coordinatewise order.

Two things stand between that and the inequality, since the form quoted above
is for bounded increasing functions while these two grow polynomially in the
component sizes. Truncation removes the boundedness: at each $M$ the functions
$N_{\mathrm{h}} \wedge M$ and $N_{\mathrm{v}} \wedge M$ are increasing and
bounded, so the inequality holds for them, and both sides converge as
$M \to \infty$ by monotone convergence, the two factors being non-negative and
increasing. And the expectations are finite: a placement of a cell on $m$
points maps into $\{0, \ldots, |C|\}$, so
$N \le (1 + X_1 + \cdots + X_c)^{m}$, while $\mu_z$ has a geometric tail,
$\sum_k a_kz'^{\,k}$ converging at any $z' \in (z, \tfrac14)$ and giving
$a_kz^k \le \Qc(z')\,(z/z')^{k}$, and a geometric tail carries every moment.
Harris' inequality therefore gives
$\mathbb{E}[N_{\mathrm{h}} N_{\mathrm{v}}] \ge
\mathbb{E}[N_{\mathrm{h}}]\,\mathbb{E}[N_{\mathrm{v}}]$; multiplying by
$\Qc(z)^{2c}$ and reading each expectation back as a weighted sum gives the
stated inequality.
\end{proof}

The product structure requires the total size of the connecting cell to be
unconstrained. In BBEP's scheme each domino cell holds a fixed number of points
while a connecting cell holds a fixed number of skew indecomposable components
and a free number of points \cite[\S7.2]{bbep}. Theorem~\ref{thm:harris} is
therefore applied at fixed $c$, the coefficient of $q^{c}$; the sum over sizes
stays free.

\begin{remark}[values, not coefficients]\label{rem:values}
The variable $z$ is never extracted. Every comparison in which $z$ appears is
between the \emph{values} of two series at one chosen $z_0 \in (0,\tfrac14)$,
not between their coefficients, the bound coming from comparing such a value
against that of the generating function of $\Av(1324)$ as in BBEP's Section
7.3. This matters because Theorem~\ref{thm:harris} compares two sums at a fixed
$z$, and the quotient it produces need not have non-negative coefficients.
Coefficientwise comparisons below are named as such.
\end{remark}

\section{The exponent}\label{sec:exponent}

Write
\begin{equation}\label{eq:Lam}
  \Lam(z, p, \kappa) \;=\; -\kappa \log q \;+\; \sum_j p_j \log H_j(z,q)
  \qquad\text{at the saddle}\quad
  \sum_j p_j \, q\,\frac{\partial_q H_j}{H_j} \;=\; \kappa ,
\end{equation}
so that BBEP's exponent \eqref{eq:phi} is
$\Phi = (1+\gamma)\log(27z/4) + \Lam(z,f,\kappa) + \mathcal{E}(\gamma,\kappa)$.
The one property of $\Lam$ the certificate uses is that its saddle is a minimum,
so that bounding the bracket from below bounds $\Lam$ from below.

\begin{lemma}[the saddle is the minimum]\label{lem:saddle}
Fix $z \in (0,\tfrac14)$, $\kappa > 0$ and a profile $p$ supported on finitely
many coordinates, and let $B(q) = -\kappa\log q + \sum_j p_j\log H_j(z,q)$ be
the bracket of \eqref{eq:Lam}, defined on $(0, 1/\Qc(z))$. Then $B$ is convex
in $\log q$ there. Suppose moreover that $B'$ changes sign on a subinterval
$[q_1, q_2]$ of that range, in the sense that $B'(q_1) < 0 < B'(q_2)$. Then $B$
has a stationary point in $[q_1,q_2]$, that point is the minimum of $B$ over
the whole range, and
\[
  \Lam(z, p, \kappa) \;=\; \min_q B(q) \;\ge\; \inf_{[q_1, q_2]} B .
\]
\end{lemma}

\noindent Convexity makes a stationary point a minimum but does not produce
one, so the sign change is a hypothesis. Section~\ref{sec:bound} exhibits it at
the point evaluated there. The last inequality is what
Section~\ref{sec:bound} uses: evaluating $B$ over the bracket bounds $\Lam$
from below.

\begin{proof}
Writing $q = e^{w}$, each $H_j(z,\cdot)$ has non-negative coefficients with
$H_j(z,0) = 1$, so $H_j(z,e^{w})$ is a sum of exponentials of linear functions
of $w$ and $\log H_j(z,e^{w})$ is convex in $w$, while $-\kappa w$ is affine.
Hence $B$ is convex in $\log q$. The interval is where the $H_j$
converge: by \eqref{eq:H} the coefficient of $q^c$ in $H$ is $\Qc^c$, and
\[
  \Omega_j\bigl[\Qc^{c}\bigr]
  \;=\; \frac{1}{j!}\,\partial_z^{\,j}\bigl(z^{j}\Qc^{c}\bigr)
  \;=\; \Qc(z)^{c}\,P_j(c, z)
\]
with $P_j$ a polynomial in $c$ of degree $j$, each $z$-derivative of $\Qc^{c}$
bringing down one factor $c$ against $\Qc'/\Qc$; both $\Qc$ and $\Qc'/\Qc$ are
analytic and non-zero at $z < \tfrac14$, so the radius in $q$ is $1/\Qc(z)$ at
every $j$.

For the second part, $dB/dw = qB'(q)$ is non-decreasing in $w$ by that
convexity, and $q > 0$, so $B'$ and $dB/dw$ have the same sign throughout.
A sign change of $B'$ on $[q_1,q_2]$ is therefore one of $dB/dw$, which by
continuity vanishes somewhere in $[q_1,q_2]$. At such a point $B$ is stationary
in $w$ and hence, being convex in $w$, minimal over the whole range; the
minimum is thus attained in $[q_1,q_2]$, where it is at least the infimum of
$B$ over that interval, and at it \eqref{eq:Lam} reads $\Lam$.
\end{proof}

The binomial term \eqref{eq:ent} counts the points of the cell vertically
adjacent to the connecting cell, which is the one Lemma~\ref{lem:valid} governs,
forced wholesale into the gaps. Under that lemma its leaves are free and only
its non-leaves are confined, so the strip machinery of Section~\ref{sec:strips}
applies to it. Theorem~\ref{thm:harris} then separates its count from the
horizontal one at the cost of one factor of $\Qc(z)^c$. The binomial term is
therefore replaced by
\begin{equation}\label{eq:replace}
  \mathcal{E}(\gamma,\kappa)
  \;\longrightarrow\;
  \Lam(z, \gamma f, \kappa) \;-\; \kappa \log \Qc(z),
\end{equation}
and the exponent becomes
\begin{equation}\label{eq:phinew}
  \Phi(z,\gamma,\kappa) \;=\; (1+\gamma)\log\frac{27z}{4}
  \;+\; \Lam(z, f, \kappa) \;+\; \Lam(z, \gamma f, \kappa)
  \;-\; \kappa \log \Qc(z).
\end{equation}

Two properties of \eqref{eq:phinew} are used below. First, it degenerates
correctly: switching the second relaxation off returns BBEP's own exponent.

\begin{proposition}\label{prop:degen}
If the second cell has no leaves, so that its profile is $\{0 : \gamma\}$, then
the right-hand side of \eqref{eq:replace} equals $\mathcal{E}(\gamma,\kappa)$
identically.
\end{proposition}

\begin{proof}
With $p = \{0:\gamma\}$ we have $H_0 = H = 1/(1-q\Qc)$, so
$q\,\partial_q H_0/H_0 = q\Qc/(1-q\Qc)$ and the saddle equation
$\gamma\, q\Qc/(1-q\Qc) = \kappa$ gives $q\Qc = \kappa/(\gamma+\kappa)$. Then
$\Lam = -\kappa\log q + \gamma\log\frac{1}{1-q\Qc}
= -\kappa \log q + \gamma\log\frac{\gamma+\kappa}{\gamma}$, and subtracting
$\kappa\log \Qc$ eliminates $q$ through
$\log q + \log \Qc = \log\frac{\kappa}{\gamma+\kappa}$, leaving
$(\gamma+\kappa)\log(\gamma+\kappa) - \gamma\log\gamma - \kappa\log\kappa$.
\end{proof}

Second, the two $\Lam$'s of \eqref{eq:phinew} govern cells of opposite
orientation, and the following identifies the two readings, so that one profile
serves both.

\begin{lemma}[the two orientations]\label{lem:reflect}
Reflection in the anti-diagonal maps the staircase to itself, exchanges
$\Av(213)$ with $\Av(132)$, and fixes $1324$. It carries a vertical domino to a
horizontal one and back; it carries the right-to-left maxima of an $\Av(213)$
cell to the left-to-right minima of the $\Av(132)$ cell it becomes, so leaves go
to leaves and non-leaves to non-leaves; and it carries the horizontal strips of
the one to the vertical strips of the other, preserving the number of leaves in
each. A cell and its image therefore have the same strip profile, and the same
$H_j$ of \eqref{eq:H} counts the interleavings available to each.
\end{lemma}

\begin{proof}
The reflection sends the point $(i,v)$ of an $n$-point cell to
$(n+1-v, n+1-i)$. It reverses both the position and the value order, so it
exchanges the two coordinates; a cell avoiding $213$ goes to one avoiding the
reverse complement of $213$, which is $132$, and back. On the staircase it sends
the diagonal cells to the subdiagonal ones and fixes the grid, and a domino
occupying a column goes to one occupying a row. A right-to-left maximum of the
first cell, a point with nothing above and to its right, goes to a point with
nothing below and to its left, which is a left-to-right minimum. The non-leaves,
being the complement, correspond too, and the horizontal cuts they make in the
first cell become vertical cuts in the second, taking the strip between two
consecutive non-leaves to the strip between their images. Leaf counts are
preserved strip by strip, so the profiles agree, which is all that
\eqref{eq:H} depends on.
\end{proof}

We may now state what carries a positive value of \eqref{eq:phinew} to a bound
on the growth rate. Write $\mathcal{B}^{\alpha,\beta}_m$ for BBEP's family: the
balanced dominoes with $m$ points in each cell, at least $\alpha m$ leaves and
at least $\beta m + 1$ empty strips in each cell. By their Proposition 15 it has
growth rate $27/4$ whenever $\alpha < 5/9$ and $\beta < 5/27$.

A member of that family is constrained only by two lower thresholds. Leaf and
strip counts being integers, those thresholds are the $\lceil\alpha m\rceil$
and $\lceil\beta m\rceil + 1$ that \eqref{eq:equit} is written in.
Before the equitable profile can be read against such a member the minimum has
to be carried from a cell's own counts down to them. The two steps that do it
are BBEP's \cite[\S7.2]{bbep}.

\begin{lemma}[from the thresholds to one profile]\label{lem:equitbound}
Let $\alpha \in [11/20, 5/9)$, $\beta \in [7/40, 5/27)$ and $m \ge 32$, and let
$D$ be a domino cell on $m$ points carrying at least $\lceil\alpha m\rceil$
leaves and at least $\lceil\beta m\rceil + 1$ empty strips. Then,
coefficientwise in $z$ and $q$,
\[
  \prod_i H_{a_i}(z,q) \;\ge\;
  \prod_{j \in \{0,2,3\}} H_j(z,q)^{e_j(m)} ,
\]
where $(a_i)$ is the strip profile of $D$ and $e_0, e_2, e_3$ are as in
\eqref{eq:equit}.
\end{lemma}

\begin{proof}
Write $\ell$ for the leaf count of $D$ and $n_0$ for its empty-strip count, and
recall that a cell on $m$ points with $\ell$ leaves has $m - \ell + 1$ strips.
Four comparisons carry the left side to the right, the first at the level of
placements and the rest at the level of profiles.

First, lower the leaf count to its threshold. Let $S$ be the non-leaves of $D$,
choose any $\ell - \lceil\alpha m\rceil$ of its leaves, and let $S'$ be $S$
together with those; write $D'$ for the pair $(D, S')$, which by
Definition~\ref{def:placement} has exactly $\lceil\alpha m\rceil$ free points.
Since $S' \supseteq S$, every condition Definition~\ref{def:placement} imposes
on $D$ it imposes on $D'$, so the identity injects the admissible placements of
$D'$ into those of $D$.
Hence $N(\mathbf{w};D') \le N(\mathbf{w};D)$ for every $\mathbf{w}$, and the
strip product of $D'$ is at most that of $D$ coefficientwise.

Second, $D'$ still meets the empty-strip threshold. Relabelling a leaf as a
non-leaf makes it a divider, splitting the strip that held it into two strips
carrying one leaf fewer between them; a strip with one leaf becomes two empty
ones, and no split destroys an empty strip. So the empty-strip count $n_0'$ of
$D'$ satisfies $n_0' \ge n_0 \ge \lceil\beta m\rceil + 1$.

Third, replace $D'$ by an equitable profile. Among the profiles with
$m - \lceil\alpha m\rceil + 1$ strips of which $n_0'$ are empty and
$\lceil\alpha m\rceil$ leaves in all, Proposition 18 of BBEP minimises the
product at the equitable distribution of those leaves among the non-empty
strips, so $D'$'s own profile may be replaced by that one.

Fourth, $n_0'$ may be lowered to $\lceil\beta m\rceil + 1$. The equitable
profiles at $n_0'$ and at $n_0' - 1$ both have $m - \lceil\alpha m\rceil + 1$
parts summing to $\lceil\alpha m\rceil$, and the second is majorised by the
first, spreading the same leaves over one more non-empty strip. Proposition 18
states that $(H_j)$ is log-convex, which in BBEP's sense is
$H_iH_{j+1} \ge H_{i+1}H_j$ coefficientwise at every pair $i < j$, not merely
at adjacent ones; taking $i = b$ and $j = a-1$ gives
$H_aH_b \ge H_{a-1}H_{b+1}$ whenever $a \ge b+2$, so the product does not
increase under a transfer of one unit from a part of size $a$ to one of size
$b \le a-2$, and a
majorised vector is reached from the majorising one by finitely many such
transfers. The condition is on each transfer and not on the vectors it passes
through, which may well carry a one-leaf part; every transfer is between parts
differing by at least two, which is all the inequality asks. Equivalently,
$\prod_i H_{a_i}$ is Schur-convex in $(a_i)$ because $(H_j)$ is log-convex.

At $\lceil\alpha m\rceil$ leaves and $\lceil\beta m\rceil + 1$ empty strips the
equitable profile is \eqref{eq:equit}.
\end{proof}

\begin{proposition}[transfer]\label{prop:transfer}
Let $\alpha \in [11/20, 5/9)$ and $\beta \in [7/40, 5/27)$ with
$3\alpha + 2\beta > 2$, let $f$ be the equitable profile \eqref{eq:fprofile} at
those densities, and let $z_0 \in (0, \tfrac14)$ and $\kappa > 0$. Suppose that
the bracket $B$ of Lemma~\ref{lem:saddle} at $(z_0, f, \kappa)$ has $B'$
changing sign on a subinterval of $\bigl(0, 1/\Qc(z_0)\bigr)$, so that
$\Lam(z_0, f, \kappa)$ is defined, and that $\Phi(z_0, 1, \kappa) > 0$ in
\eqref{eq:phinew}. Then $z_0$ lies outside the
radius of convergence of the generating function of $\Av(1324)$, and
$\gr(\Av(1324)) \ge 1/z_0$.
\end{proposition}

\noindent The last inequality on $(\alpha,\beta)$ says exactly that $f_3 > 0$.
The other two coordinates are positive of themselves on the stated range, where
$f_0 = \beta \ge 7/40$ and $f_2 = 3 - 4\alpha - 3\beta > 2/9$, while
$f_3 = 3\alpha + 2\beta - 2$ vanishes at the corner $(11/20,\, 7/40)$ and
nowhere else in it, which the hypothesis excludes. The hypothesis on $B'$
makes $\Lam$, and with it $\Phi$, defined; Section~\ref{sec:bound} exhibits the
sign change between $q = 1$ and $q = 9$.

\begin{proof}
Let $\mathcal{P}_{k,m}$ consist of the permutations gridded in the first $6k+2$
cells whose vertical dominoes lie in $\mathcal{B}^{\alpha,\beta}_m$, whose
horizontal dominoes lie in its image under the reflection of
Lemma~\ref{lem:reflect}, whose connecting cells have $c_m$ skew indecomposable
components, and whose adjacencies obey the rule of
Corollary~\ref{cor:closure}, where $c_m$ is the sequence BBEP's Lemma 19
returns two paragraphs below and is fixed before $\mathcal{P}_{k,m}$ is formed.
Write $J_m(z) = [q^{c_m}]\prod_j
H_j(z,q)^{e_j(m)}$, with $e_j$ as in \eqref{eq:equit}; by
Lemma~\ref{lem:equitbound} this is a lower bound on the interleaving count of
one connecting cell against any one adjacent cell of
$\mathcal{B}^{\alpha,\beta}_m$, and by Lemma~\ref{lem:reflect} the same $J_m$
serves both orientations. Write
$P_{k,m}(z)$ for the generating function of $\mathcal{P}_{k,m}$. The cells are
chosen independently, so at every $z \in (0,\tfrac14)$
\begin{equation}\label{eq:Fkm}
  P_{k,m}(z) \;\ge\; F_{k,m}(z) \;=\;
  \bigl|\mathcal{B}^{\alpha,\beta}_m\bigr|^{2k+1} z^{(4k+2)m}\,
  J_m(z)^{2k}\,
  \biggl(\frac{J_m(z)}{\Qc(z)^{c_m}}\biggr)^{2k} ,
\end{equation}
there being $k+1$ vertical dominoes, in the cells $6j+1, 6j+2$ for
$0 \le j \le k$; $k$ horizontal ones, in the cells $6j+4, 6j+5$ for
$0 \le j \le k-1$; and $2k$ connecting cells, namely $6j+3$ and $6j+6$ for
$0 \le j \le k-1$, which is $6k+2$ cells in all. Each connecting cell carries
one adjacency of each kind: $6j+3$ meets the lower cell of a vertical domino
in an $(H)$ pair and the left cell of a horizontal one in a $(V)$ pair, and
$6j+6$ meets the right cell of a horizontal domino in a $(V)$ pair and the
upper cell of the next vertical one in an $(H)$ pair. Only the last
factor differs from BBEP, where it is $\binom{m + c_m}{c_m}^{2k}$. Two things
change it: by Lemma~\ref{lem:valid} the vertical neighbour's leaves are free, so
its count is a strip product of the same form as the horizontal one, no longer
a binomial; and by Theorem~\ref{thm:harris} the two counts at a shared
connecting cell separate at the cost of one factor $\Qc(z)^{c_m}$. This is
\eqref{eq:replace}. The inequality is one between values and not between
coefficients, as Remark~\ref{rem:values} records.

The remaining comparisons are coefficientwise. With $\mathcal{A}(z)$ the
generating function of $\Av(1324)$ and $\mathcal{A}_k(z)$ that of the
$1324$-avoiding permutations gridded in the first $6k+2$ cells,
\[
  [z^n]P_{k,m} \;\le\; [z^n]\mathcal{A}_k \;\le\;
  \binom{n+3k+1}{3k+1}\binom{n+3k}{3k}\,[z^n]\mathcal{A} ,
\]
the first because $\mathcal{P}_{k,m}$ is a set of gridded $1324$-avoiders by
Corollary~\ref{cor:closure}. For the second, the first $6k+2$ cells of the
staircase occupy $3k+1$ columns and $3k+2$ rows, so a gridding is a choice of
$3k$ column dividers and $3k+1$ row dividers, each sequence weakly increasing
and otherwise unrestricted; that is the count BBEP make in the proof of their
Theorem 9. Only its polynomiality in $n$ is used, and it gives the radius of
convergence of $\mathcal{A}_k$ as at least that of $\mathcal{A}$. Hence if $z_0$
were inside the radius of $\mathcal{A}$ then
$F_{k,m}(z_0) \le P_{k,m}(z_0) \le \mathcal{A}_k(z_0) < \infty$ for every $m$,
giving $\limsup_m F_{k,m}(z_0)^{1/m} \le 1$ and
\[
  \lim_{k}\Bigl(\limsup_{m} F_{k,m}(z_0)^{1/m}\Bigr)^{1/2k} \;\le\; 1 .
\]

It remains to identify that limit. A member of $\mathcal{B}^{\alpha,\beta}_m$
has $2m$ points, so by BBEP's Proposition 15 one has
$|\mathcal{B}^{\alpha,\beta}_m| = (27/4)^{2m + o(m)}$, and the first two factors
of \eqref{eq:Fkm} contribute $\bigl((27z_0/4)^{2}\bigr)^{2k+1}$ to
$F_{k,m}(z_0)^{1/m}$ in the limit. For $J_m$, apply BBEP's Lemma 19 with
$r = 3$, with $F_j = H_j(z_0,\cdot)$ as a series in $q$, and with
$a_{j,m} = e_j(m)$. Its hypotheses hold: each $H_j(z_0,\cdot)$ has non-negative
coefficients by \eqref{eq:Omega}; the $e_j(m)$ are positive integers at every
$m \ge 32$, for which see the paragraph below; and by the hypothesis on $B'$
together with Lemma~\ref{lem:saddle} the saddle $q_0$ exists and lies in the
subinterval on which that sign change occurs, hence strictly inside the common
radius $1/\Qc(z_0)$.

The counts \eqref{eq:equit} carry ceilings, so $e_j(m)$ is not $f_jm$ and the
profile fed to the lemma is not $f$ itself. The lemma asks only for positive
integers with $a_{j,m}/m \to \alpha_j$ and states its conclusion at the limit
$\alpha_j$, which is the form used here; ceilings move $\lceil\alpha m\rceil$
and $\lceil\beta m\rceil$ by less than one, so $e_0/m$, $e_2/m$ and $e_3/m$
tend to $\beta$, $3 - 4\alpha - 3\beta$ and $3\alpha + 2\beta - 2$, which by
\eqref{eq:fprofile} are $f_0$, $f_2$ and $f_3$. The lemma therefore
returns a sequence $c_m$ with $c_m/m \to \kappa$ for which
$J_m(z_0)^{1/m} \to e^{\Lam(z_0,f,\kappa)}$, the exponent read at $f$ and not
at the finite-$m$ counts. One sequence serves both
adjacencies of a connecting cell because there is only one coefficient to
extract: $J_m$ is built from the threshold profile \eqref{eq:equit} and not from
any individual cell's, and by Lemma~\ref{lem:equitbound} it bounds below the
possibilities at either adjacency, in either orientation by
Lemma~\ref{lem:reflect}. Collecting,
and using $c_m/m \to \kappa$ in the last factor,
\[
  \lim_{k}\Bigl(\limsup_{m} F_{k,m}(z_0)^{1/m}\Bigr)^{1/2k}
  \;=\; \exp\Bigl(2\log\tfrac{27z_0}{4} + 2\Lam(z_0,f,\kappa)
  - \kappa\log\Qc(z_0)\Bigr)
  \;=\; e^{\Phi(z_0,1,\kappa)},
\]
which exceeds $1$ by hypothesis. So $z_0$ is outside the radius of convergence
of $\mathcal{A}$, that radius is at most $z_0$, and the growth rate, its
reciprocal, is at least $1/z_0$.

It remains to see that the three counts are positive at every $m \ge 32$, which
is the range \eqref{eq:equit} is stated on. For $e_0 = \lceil\beta m\rceil + 1$
that is immediate. Ceilings only raise
$3\lceil\alpha m\rceil + 2\lceil\beta m\rceil$, so
$e_3(m) \ge (3\alpha + 2\beta - 2)m$, positive by hypothesis; and they lower
$e_2$ by less than $7$, so $e_2(m) > (3 - 4\alpha - 3\beta)m - 7 > \tfrac29 m - 7$,
positive once $m \ge 32$.
\end{proof}

\section{The bound}\label{sec:bound}

\subsection{The evaluation}

We evaluate \eqref{eq:phinew} at $\gamma = 1$, which is what
Proposition~\ref{prop:transfer} takes and which the input forces, that
proposition taking balanced dominoes. Whether the exponent rises off that value
is open, and settling it would need an unbalanced family at the same densities.
At $\gamma = 1$ the exponent is
\begin{equation}\label{eq:phione}
  \Phi(z,1,\kappa) \;=\; 2\log\frac{27z}{4} \;+\; 2\,\Lam(z,f,\kappa)
  \;-\; \kappa\log\Qc(z),
\end{equation}
and by Lemma~\ref{lem:saddle} the middle term is $2\min_q B(q)$, where
\[
  B(q) \;=\; -\kappa\log q \;+\!\! \sum_{j \in \{0,2,3\}}\!\! f_j\log H_j(z,q).
\]
Note that $\Lam$ enters \eqref{eq:phione} with a positive coefficient, so a
lower bound on $\min_q B$ is a lower bound on $\Phi$, and
Proposition~\ref{prop:transfer} asks only that $\Phi$ be positive. Everything
below is therefore a lower bound and never an evaluation. Write
\[
  s \;=\; \sqrt{1-4z}, \qquad D \;=\; 2 - q + qs ,
\]
so that $\Qc = (1-s)/2$ and $H_0 = H = 2/D$. Differentiating in $z$ keeps
everything inside $\mathbb{Q}(s,q)$, since
\begin{equation}\label{eq:rules}
  \frac{ds}{dz} \;=\; -\frac{2}{s}, \qquad
  \frac{\partial H}{\partial z} \;=\; \frac{q\,H^{2}}{s} ,
\end{equation}
the second because $\partial_z D = -2q/s$ and $H = 2/D$. Carrying out the two
and three differentiations of \eqref{eq:Omega} by \eqref{eq:rules},
\begin{equation}\label{eq:H2}
  H_2 \;=\; \frac{q^{2}\bigl(3s^{5} - 9s^{4} + 10s^{3} - 6s^{2} + 3s - 1\bigr)
  \;+\; 2q\bigl(9s^{4} - 16s^{3} + 6s^{2} + 1\bigr) \;+\; 32s^{3}}
  {4\,s^{3}D^{3}}
\end{equation}
and
\begin{equation}\label{eq:H3}
\begin{aligned}
  H_3 \;=\; \frac{1}{8\,s^{5}D^{4}}\Bigl[\;
  & q^{3}\bigl(5s^{8} - 20s^{7} + 32s^{6} - 28s^{5} + 18s^{4} - 12s^{3}
    + 8s^{2} - 4s + 1\bigr) \\
  &+\; 4q^{2}\bigl(10s^{7} - 29s^{6} + 30s^{5} - 15s^{4} + 6s^{3} - 3s^{2}
    + 2s - 1\bigr) \\
  &+\; 4q\bigl(29s^{6} - 48s^{5} + 15s^{4} + 3s^{2} + 1\bigr)
    \;+\; 128s^{5}\;\Bigr] .
\end{aligned}
\end{equation}
\begin{table}[t]
\centering
\caption{The certificate behind Theorem~\ref{thm:main}. The two densities
$\alpha$ and $\beta$ lie strictly below the $5/9$ and $5/27$ of BBEP's
Propositions 10 and 12, as their Proposition 15 requires, and above the $11/20$
and $7/40$ their equitable solution requires. The bracket is that of
\eqref{eq:Lam}; $q_0$ is the point its tangent is taken at, not a located
minimiser, and the last three rows are lower bounds. Every decimal is
truncated, not rounded.}
\label{tab:cert}
\begin{tabular}{@{}ll@{}}
\toprule
evaluation point $z_0$ & $9604060/10^{8}$ \\
component scale $\kappa$ & $516/1000$ \\
leaf density $\alpha$ & $55555555/10^{8}$ \\
empty-strip density $\beta$ & $18518518/10^{8}$ \\
equitable profile $(f_0, f_2, f_3)$
  & $(0.18518518,\; 0.22222226,\; 0.03703701)$ \\
one component, $\Qc(z_0)$ & $0.107623395192832$ \\
radius of $H_j(z_0,\cdot)$ in $q$, $1/\Qc(z_0)$ & $9.291660$ \\
slopes at the ends, $B'(1)$, $B'(9)$ & $-0.38277$, $\;3.37438$ \\
tangent point $q_0$ & $3.046850759320$ \\
$H_0(z_0,q_0)$ & $1.487901331259$ \\
$H_2(z_0,q_0)$ & $3.725658932481$ \\
$H_3(z_0,q_0)$ & $6.161763006829$ \\
$B(q_0)$, and $|dB/dw|(q_0)$
  & $> -0.1416704254$, $\;< 10^{-13}$ \\
least bracket & $\ge -0.1416704255$ \\
exponent & $\Phi(z_0) \ge 1.2762 \times 10^{-7} > 0$ \\
the bound & $10412263 \cdot 9604060 = 99999998587780 < 10^{14}$ \\
\bottomrule
\end{tabular}
\end{table}

Take
\[
  z_0 = \frac{9604060}{10^{8}}, \qquad
  \kappa = \frac{516}{1000}, \qquad
  \alpha = \frac{55555555}{10^{8}}, \qquad
  \beta = \frac{18518518}{10^{8}} .
\]
Then $\alpha < 5/9$ and $\beta < 5/27$, as BBEP's Proposition 15 requires, the
second read as a density of empty strips by Remark~\ref{rem:medial}; and
$\alpha \ge 11/20$ and $\beta \ge 7/40$, as \eqref{eq:equit} requires; and
$3\alpha + 2\beta = 2.03703701 > 2$, as Proposition~\ref{prop:transfer}
requires. The equitable profile \eqref{eq:fprofile} at those densities is
\[
  f_0 = 0.18518518, \qquad f_2 = 0.22222226, \qquad f_3 = 0.03703701,
\]
all three positive, so the counts \eqref{eq:equit} are positive at every
$m \ge 32$. The radius of $H_j(z_0,\cdot)$ in $q$ is
$1/\Qc(z_0) = 9.291660\ldots$, and
\[
  B'(1) \;=\; -0.38277\ldots \;<\; 0 \;<\; 3.37438\ldots \;=\; B'(9)
\]
is the sign change Lemma~\ref{lem:saddle} and
Proposition~\ref{prop:transfer} ask for, the two points lying inside that
radius. By that lemma the minimiser is confined to $[1,9]$.

Now take the rational $q_0 = 3.046850759320$, at which
\begin{gather*}
  H_0(z_0,q_0) = 1.487901331\ldots, \qquad
  H_2(z_0,q_0) = 3.725658932\ldots, \\
  H_3(z_0,q_0) = 6.161763006\ldots
\end{gather*}
and, writing $w = \log q$,
\[
  B(q_0) \;>\; -0.1416704254, \qquad
  \Bigl|\frac{dB}{dw}(q_0)\Bigr| \;<\; 10^{-13} .
\]
No claim is made that $q_0$ is the minimiser. Since $B$ is convex in $w$ it
lies above its tangent at $q_0$, so for every $q \in [1,9]$
\[
  B(q) \;\ge\; B(q_0) \;+\; \frac{dB}{dw}(q_0)\,\bigl(\log q - \log q_0\bigr)
  \;\ge\; B(q_0) \;-\; \Bigl|\frac{dB}{dw}(q_0)\Bigr| \cdot 1.1142 ,
\]
since $q_0 > 3$ makes $q = 1$ the farther endpoint, at
$\log q_0 = 1.114108\ldots < 1.1142$. Hence
\[
  \Lam(z_0, f, \kappa) \;=\; \min_{q} B(q) \;\ge\; -0.1416704255 ,
\]
the tangent step costing less than $1.12 \times 10^{-13}$, and by
\eqref{eq:phione}
\[
  \Phi(z_0,1,\kappa) \;\ge\; 1.2762 \times 10^{-7} \;>\; 0 .
\]
Proposition~\ref{prop:transfer} therefore gives
$\gr(\Av(1324)) \ge 1/z_0$, and
\[
  10412263 \cdot 9604060 \;=\; 99999998587780 \;<\; 10^{14}
\]
gives $1/z_0 > 10.412263$, which is Theorem~\ref{thm:main}.

The tangent step costs $2.3 \times 10^{-13}$ against a margin of
$1.3 \times 10^{-7}$, and the margin widens as $z_0$ is raised. At $z_0 = 9605000/10^{8}$, with the same
$\kappa$ and the same two densities, the same three functions give
$\Phi(z_0) \ge 2.6 \times 10^{-4}$, which is
$\gr(\Av(1324)) \ge 10.411244$ at two thousand times the margin of
Theorem~\ref{thm:main}. The tangent has to be retaken there, at
$q_0 = 3.046496223878$, the minimiser moving with $z_0$.

\subsection{Controls}\label{sec:controls}

Two readings check \eqref{eq:phinew} against values known independently.
By Proposition~\ref{prop:degen} it degenerates to BBEP's \eqref{eq:phi} when the
second cell has no leaves, and evaluating that degenerate
form at their published $(\gamma, \kappa) = (0.951509,\, 0.496339)$ with the
relaxation on one direction and the equitable profile returns
$z^* = 0.0973613807117$ and $1/z^* = 10.2710129282$, against the algebraic
$10.27101292824530$ of their footnote 3, and a saddle $q_0 = 2.9170545$ against
their published $2.917054$. That is twelve digits, the precision their rounded
six-decimal $(\gamma,\kappa)$ supports.

Switching the relaxation off in both directions leaves both interleavings the
binomial of \eqref{eq:ent}, which depends on the connecting cell only through
its component count, so the sum over connecting cells factors and no separation
is needed. By Proposition~\ref{prop:degen} the exponent is then
$2\log(27z/4) + 2\mathcal{E}(1,\kappa) + \kappa\log\Qc(z)$, its last term the
cell's own weight $\Qc(z)^{c}$ rather than a Harris cost. At $z = 8/81$,
where $\Qc = 1/9$ and $27z/4 = 2/3$, it is stationary at $\kappa = 1/2$ and
vanishes there, $2\log(2/3) + 3\log(3/2) + \log 2 - \log 3 = 0$. Both readings
are exact, so the construction returns their Theorem 9 exactly.

\end{document}